\documentclass{amsart}
\usepackage{latexsym,amsmath}
\usepackage{amsthm,amssymb,mathrsfs}

\newtheorem{theorem}{Theorem}[section]
\newtheorem{lemma}[theorem]{Lemma}
\newtheorem{corollary}[theorem]{Corollary}
\newtheorem{proposition}[theorem]{Proposition}

\theoremstyle{definition}
\newtheorem{definition}[theorem]{Definition}

\newtheorem{remark}[theorem]{Remark}
\newtheorem{question}[theorem]{Question}

\numberwithin{equation}{section}

\def\RR{\mathbb{R}} % the real number R
\def\CC{\mathbb{C}} % Complex C
\def\NN{\mathbb{N}} % Natural number N
\def\im{\sqrt{-1}} % the imaginary number i
\def\ddbar{\partial\bar\partial} %
\newcommand{\paren}[1]{\left(#1\right)}

\newcommand{\pd}[2]{\frac{\partial#1}{\partial#2}}
\newcommand{\abs}[1]{\left\vert#1\right\vert}
\newcommand{\norm}[1]{\left\|#1\right\|}

\newcommand{\inner}[1]{\left\langle{#1}\right\rangle}

\begin{document}

\title[Subharmonicity of K\"ahler-Einstein metrics]{Subharmonicity of variations of K\"ahler-Einstein metrics on pseudoconvex domains}

\author{Young-Jun Choi}

\address{School of Mathematics, Korea Institute for Advanced Study(KIAS), 85 Hoegiro Dongdaemun-gu, Seoul 130-722 790-784, Republic of Korea}

\email{choiyj@kias.re.kr}



\maketitle

\begin{abstract}
This paper is a sequel to \cite{Choi} in Math. Ann. In that paper we studied the subharmonicity of K\"ahler-Einstein metrics on strongly pseudoconvex domains of dimension greater than or equal to $3$. In this paper, we study the variations K\"ahler-Einstein metrics on bounded strongly pseudoconvex domains of dimension $2$. In addition, we discuss the previous result with general bounded pseudoconvex domain and local triviality of a family of bounded strongly pseudoconvex domains.
\end{abstract}

\section{Introduction}

Let $(z,s)\in\CC^n\times\CC$ be the standard coordinates and $\pi:\CC^n\times\CC\rightarrow\CC$ be the projection on the second factor. Let $D$ be a smooth domain in $\CC^{n+1}$ such that for each $s\in\pi(D)$, the slice $D_s=D\cap\pi^{-1}(s)=\{z:(z,s)\in{D}\}$ is a bounded strongly pseudoconvex domian with smooth boundary.

In \cite{Cheng_Yau}, Cheng and Yau constructed a unique complete K\"ahler-Einstein metric on a bounded strongly pseudoconvex domain with smooth boundary. This implies that there exists a unique complete K\"abler-Einstein metric $h_{\alpha\bar\beta}(z,s):=h^s_{\alpha\bar\beta}(z)$ on each slice $D_s$ which satisfies the following:
\begin{align*}
-(n+1)h_{\alpha\bar\beta}(z,s)
	& =\mathrm{Ric}_{\alpha\bar\beta}(z,s) \;\;\;\;\;\;\text{(the Ricci tensor)} \\
	& =-\pd{^2}{z^\alpha\partial z^{\bar\beta}}
	\log\det\paren{h_{\gamma\bar\delta}(z,s)}_{1\le\gamma,\delta\le{n}}.
\end{align*}
Namely, the Ricci curvature is a negative constant $-(n+1)$. This constant could be any negative number; $-(n+1)$ is chosen for convenience. On each slice $D_s$, 
\begin{equation*}
h(z,s)
:=
\frac{1}{n+1}\log\det\paren{h_{\gamma\bar\delta}(z,s)}_{1\le\gamma,\delta\le{n}}
\end{equation*}
is a potential function of the K\"{a}hler-Einstein metric $h_{\alpha\bar\beta}(\cdot,s)$. We can consider $h$ as a smooth function on $D$ (\cite{Choi}). It is an immediate consequence of the K\"{a}hler-Einstein conditions that the restriction of $h$ to each slice $D_s$ is strictly plurisubharmonic. But it is not obvious that it is also plurisubharmonic or strictly plurisubharmonic in the base direction (the $s$-direction). In \cite{Choi}, we have shown that if the slice dimension $n$ is greater than or equal to $3$, then $h$ is plurisubharmonic provided $D$ is pseudoconvex. Moreover, we have also proved that $h$ is strictly plurisubharmonic if $D$ is strongly pseudoconvex.
\medskip

In this paper, we shall deal with a family of bounded smooth strongly pseudoconvex domain of dimension greater than or equal to $2$. It is remarkable to note that Maitani and Yamaguchi already proved the $1$-dimensional slice case (\cite{Maitani_Yamaguchi}).

\begin{theorem}\label{T:main_theorem}
With the above notations, if $D$ is a strongly pseudoconvex domain in $\CC^{n+1}$, then $h(z,s)$ is a strictly plurisubharmonic function.
\end{theorem}

In case of a general bounded pseudoconvex domain, Cheng and Yau also constructed a unique K\"{a}hler-Einstein metric which is almost complete, which is a limit of K\"ahler-Einstein metrics on relatively compact subdomains (\cite{Cheng_Yau}). In \cite{Mok_Yau}, Mok and Yau proved that this metric is, in fact, complete. Hence we can consider the situation that $D$ is a pseudoconvex domain such that each slice $D_s$ is a bounded pseudoconvex domain. By simple approximation process, we have the following corollary.

\begin{corollary}\label{C:subharmonicity}
Under the above hypothesis, $h$ is a plurisubharmonic function.
\end{corollary}

In \cite{Tsuji2}, Tsuji showed a dynamical construction of a K\"{a}hler-Eistein metric on a bounded strongly pseudoconvex domain with smooth boundary. More precisely, he have shown that the K\"ahler-Einstein metric is the iterating limit of the Bergman metric. Using the Berndtsson's result (\cite{Berndtsson1}), he proved the same result with Corollary~\ref{C:subharmonicity}.
\medskip

The above setting is also considered as a family of bounded strongly pseudoconvex domains. Moreover, the geodesic curvature (which is defined in Section~\ref{S:Prelimiaries}) is strongly related with the Kodaira-Spencer map. So it is natural to ask what happens if the geodesic curvature vanishes. The following theorem answers this question.

\begin{theorem} \label{T:triviality}
Suppose that the slice dimension $n$ is greater than or equal to $3$. If the geodesic curvature vanishes, then the family is locally trivial.
\end{theorem}

The proof of Theorem \ref{T:triviality} depends the vanishing order of the solution of complex Monge-Amp\`ere equation near the boundary. This is why our method is not applicable to the case that the slice dimension $n=2$.
\medskip

We will all the time consider only the case of a one dimensional base, but the computations are easily generalized  to the case of a higher dimensional base. Throughout this paper we use  small Greek letters, $\alpha,\beta,\dots=1,\dots,n$ for indices on $z\in\CC^n$ unless otherwise specified. For a properly differentiable function $f$ on $\CC^n\times\CC$, we denote by
\begin{equation*} 
f_\alpha
=
\pd{f}{z^\alpha}
\;\;\;
\textrm{and}
\;\;\;
f_{\bar\beta}
=
\pd{f}{z^{\bar\beta}}.
\end{equation*}
where $z^{\bar\beta}$ mean $\overline{z^\beta}$. If there is no confusion, we always use the Einstein convention. For a complex manifold $X$, we denote by $T'X$ the complex tangent vector bundle of $X$ of type $(1,0)$.

\section{Prelimiaries}\label{S:Prelimiaries}
In this section, we recaptulate the result in \cite{Choi}. Throughout this section, $D$ is a smooth domain in $\CC^{n+1}$ such that every slice 
$$
D_s=D\cap\pi^{-1}(s)=\{z:(z,s)\in{D}\}
$$ 
is a bounded strongly pseudoconvex domain with smooth boundary. Since our computation is always local in $s$-variable, we may assume that $\pi(D)=U$ the standard unit disc in $\CC$.

\subsection{Horizontal lifts and Geodesic curvatures}

\begin{definition}
Let $\tau$ be a real $(1,1)$-form on $D$ which is positive definite on each slice $D_s$. We denote by $v:=\partial/\partial s$ the holomorphic coordinate vector field.
\begin{itemize}
\item[1.] A vector field $v_\tau$ of type $(1,0)$ is  called a \emph{horizontal lift} along $D_s$ of $v$ if $v_\tau$ satisfies the following:
\begin{itemize}
\item [(\romannumeral1)]$\inner{v_\tau,w}_\tau=0$ for all $w\in{T'D_s}$,
\item [(\romannumeral2)]$d\pi(v_\tau)=v$.
\end{itemize}
\item[2.] The \emph{geodesic curvature} $c(\tau)$ of $\tau$ is defined by the norm of $v_\tau$ with respect to the sesquilinear form $\inner{\cdot,\cdot}_\tau$ induced by $\tau$, namely,
\begin{equation*}
c(\tau)=\inner{v_\tau,v_\tau}_\tau.
\end{equation*}
\end{itemize}
\end{definition}
Note that under the holomorphic coordinate $(z,s)$, $\tau$ is written by
\begin{equation*}
\tau
=
\im\paren{\tau_{s\bar s}ds\wedge{d\bar s} 
+\tau_{s\bar\beta}ds\wedge{dz}^{\bar\beta} 
+\tau_{\alpha\bar s}dz^\alpha\wedge{d\bar s} 
+\tau_{\alpha\bar\beta}dz^\alpha\wedge{dz}^{\bar\beta}
}.
\end{equation*}
Then the horizontal lift $v_\tau$ and the geodesic curvature $c(\tau)$ can be written by the following:
\begin{equation*}
v_\tau
=
\pd{}{s}-\tau_{s\bar\beta}\tau^{\bar\beta\alpha}\pd{}{z^\alpha}
\;\;\;
\textrm{and}
\;\;\;
c(\tau)
=
\tau_{s\bar{s}}-\tau_{s\bar\beta}\tau^{\bar\beta\alpha}\tau_{\alpha\bar{s}}.
\end{equation*}
Then it is well known that 
\begin{equation}\label{E:slice_positive}
\frac{\tau^{n+1}}{(n+1)!}=c(\tau)\cdot\frac{\tau^n}{n!}\wedge\im{d}s\wedge{d}\bar{s}.
\end{equation}
It is remarkable to note that since $\tau$ is positive definite when restricted to $D_s$, \eqref{E:slice_positive} implies that if $c(\tau)>0 \; (\ge0)$, then $\tau$ is a positive (semi-positive) real $(1,1)$-form.

\subsection{The geodesic curvatures of the real $(1,1)$-forms induced by defining functions} \label{SS:GC}
Since every slice $D_s$ is a bounded smooth strongly pseudoconvex domain, we can take a defining function of $D$ which satisfies the following conditions:
\begin{itemize}
\item [(\romannumeral1)] $\varphi\in{C}^\infty(\bar{D})$ and $D=\{(z,s)\in\CC^{n+1}:\varphi(z,s)<0\}$,
\item [(\romannumeral2)] $\partial\varphi\neq0$ on $\partial{D}$,
\item [(\romannumeral3)] $\paren{\varphi_{\alpha\bar\beta}(\cdot,s)}>0$ in $\bar D$ and
\item [(\romannumeral4)] $\partial_z\varphi\neq0$ on $\partial D$.
\end{itemize}
We denote by $g=-\log(-\varphi)$. Then it follows that 
\begin{equation} 
g_{\alpha\bar\beta}
=\frac{\varphi_{\alpha\bar\beta}}{-\varphi}+
\frac{\varphi_\alpha\varphi_{\bar\beta}}{\varphi^2}\;
\end{equation}
and the inverse is
\begin{equation} \label{E:inverse}
g^{\bar\beta\alpha}
	=(-\varphi)\paren{\varphi^{\bar\beta\alpha}
	+\frac{\varphi^{\bar\beta}\varphi^\alpha}{\varphi-\abs{d\varphi}^2}},
\end{equation}
where $\abs{d\varphi}^2=\varphi_\alpha\varphi_{\bar\beta}g^{\alpha\bar\beta}$. By some computation, we have $g^{\alpha\bar\beta}g_\alpha g_{\bar\beta}\le1$. It follows that $g_{\alpha\bar\beta}$ gives a complete K\"ahler metric on each $D_s$ (\cite{Cheng_Yau}). Now we define the real $(1,1)$-form $G$ by $G=\im\ddbar g$. 
A direct computation gives the following:
\begin{equation*}
g_{s\bar\beta}g^{\bar\beta\alpha}
	=\varphi_{s\bar\beta}\paren{\varphi^{\bar\beta\alpha}
	+\frac{\varphi^{\bar\beta}\varphi^\alpha}{\varphi-\abs{d\varphi}^2}}
	+\frac{\varphi^\alpha\varphi_s}{\abs{d\varphi}^2-\varphi}.
\end{equation*}
This equation shows that following proposition.
\begin{proposition}[\cite{Choi}] 
Any horizontal lift $v_G$ with respect to $G$ is smoothly extended up to the boundary $\partial D$. Moreover, $v_G\vert_{\partial D}$ is tangent to  $\partial D$.
\end{proposition}

\subsection{Fefferman's approximate solutions and the boundary behavior of the solution of complex Monge-Amp\`{e}re equation} \label{SS:boundary_behavior}
Let $\Omega$ be a bounded strongly pseudoconvex domain with smooth boundary. Given a smooth function $\zeta$ on $\Omega$, we define $J(\zeta)$ by
\begin{equation*}
J(\zeta)=(-1)^n\det
\left(
\begin{array}{cc}
\zeta & \zeta_{\bar\beta} \\
\zeta_\alpha & \zeta_{\alpha\bar\beta}
\end{array}
\right).
\end{equation*}
Note that if $\zeta>0$ in $\Omega$ and $g=-\log\zeta$, then it is easy to show that
\begin{equation*}
J(\zeta)=e^{-(n+1)g}\det\paren{g_{\alpha\bar\beta}}.
\end{equation*}
Consider the following problem:
\begin{equation} \label{E:Fefferman}
\begin{aligned}
J(\zeta)=1\;\;\text{on}\;\;\Omega, \\
\zeta=0\;\;\text{on}\;\;\partial\Omega.
\end{aligned}
\end{equation}
In \cite{Fefferman}, Fefferman developed a formal technique to find approximate solutions of \eqref{E:Fefferman}:
\medskip

Let $\rho$ be a defining function of $\Omega$ such that $d\rho\neq0$ on $\partial\Omega$. We define recursively
\begin{equation} \label{E:approximate_solution}
\begin{aligned}
\rho^1 & =-\rho\cdot\paren{J(-\rho)}^{-\frac{1}{n+1}}, \\
\rho^l & =\rho^{l-1}\paren{1+\frac{1-J(\rho^{l-1})}{(n+2-l)l}}
\;\;\text{for}\;\;\;2\le{l}\le{n+1}.
\end{aligned}
\end{equation}
Then $\rho^l$ satisfies the following properties:
\begin{itemize}
\item [(1)] Every $-\rho^l$ is also a defining function of $\Omega$. In particular, we may assume that every $\rho^l$ is considered as a smooth function defined on $\CC^n$.
\item [(2)] $J(\rho^l)=1+O(\abs\rho^l)$ for $l=1,\dots,n+1$, i.e., $\rho^l$ is an approximate solution of order $l$ for $l=1,\dots,n+1$.
\end{itemize}

By \eqref{E:approximate_solution}, we can write $-\rho^l=\eta\rho$ for some $\eta\in{C}^\infty(\bar\Omega)$. Let $w=-\log(-\eta\rho)$ and $J(-\eta\rho)=e^{-F}$.  Then we have 
\begin{equation*}
\det\paren{w_{\alpha\bar\beta}}=e^{(n+1)w}e^{-F},
\end{equation*}
and
\begin{equation} \label{E:initial}
F=-\log{J}(-\eta\rho)=-\log{J}(\rho^l)=O(\abs\rho^l).
\end{equation}
Since $\eta$ is positive near $\partial\Omega$, we know that $w$ is strictly plurisubharmonic when sufficiently close to the boundary and diverges on $\partial\Omega$. By modifying $w$ away from $\partial\Omega$, we may assume that $w$ is strictly plurisubharmonic on $\Omega$. We denote it by $w$ and again write $\det\paren{w_{\alpha\bar\beta}}=e^{(n+1)w}e^{-F}$. Thus $F$ is now a smooth function on $\Omega$ and still satisfies that condition \eqref{E:initial}. Again $\eta$ is understood to be a smooth function on $\bar\Omega$ such that $w=-\log(-\eta\rho)$. 

Cheng and Yau's theorem implies that we can solve the following equation:
\begin{equation} \label{E:Monge-Ampere}
\begin{aligned}
\det(w_{\alpha\bar\beta}+u_{\alpha\bar\beta})
=e^{(n+1)u}e^F\det(w_{\alpha\bar\beta}) \\
\frac{1}{c}(w_{\alpha\bar\beta})\le(w_{\alpha\bar\beta}+u_{\alpha\bar\beta}) \le c(w_{\alpha\bar\beta}).
\end{aligned}
\end{equation}
Note that $F=(n+1)w-\log\det(w_{\alpha\bar\beta})$. This implies that
\begin{equation*}
F_{\alpha\bar\beta}=(n+1)w_{\alpha\bar\beta}+R_{\alpha\bar\beta},
\end{equation*}
where $R_{\alpha\bar\beta}$ is the component of Ricci curvature tensor of the K\"ahler metric $w_{\alpha\bar\beta}$. It follows that $\sum\paren{w_{\alpha\bar\beta}+u_{\alpha\bar\beta}}dz^\alpha{dz}^{\bar\beta}$ is the unique complete K\"{a}hler-Einstein metric in $\Omega$. Cheng and Yau also described the boundary behavior of the solution $u$ of \eqref{E:Monge-Ampere}:

\begin{theorem}[Simple Version \cite{Cheng_Yau}] \label{T:CY2}
Suppose that $\Omega$ is a smooth strongly pseudoconvex domain in $\CC^n$ and $\rho$ is a smooth defining function of $\Omega$. Suppose that $F=\xi(-\rho)^k$, $1\le{k}\le{n+1}$, $\xi\in{C^\infty(\bar\Omega)}$. Suppose that $u$ is a solution of \eqref{E:Monge-Ampere}. Then
\begin{equation*}
\abs{D^pu}(x)=O(\abs{\rho}^{a/2-p})
\end{equation*}
where $a<\min(2n+1,2k)$ and $\abs{D^pu}(x)$ is the Euclidean length of the $p$-th derivative of $u$.
\end{theorem}

Now suppose $u$ be a solution to \eqref{E:Monge-Ampere} with $w=-\log(-\rho^{n+1})=-\log(-\eta\rho)$ and $F=-\log{J}(-\eta\rho)$. Then we know that 
\begin{equation*}
F=-\log{J}(-\eta\rho)=-\log\paren{1+\xi(-\rho)^{n+1}}
\end{equation*}
for some $\xi\in C^\infty(\Omega)$. Then Theorem \ref{T:CY2} says that
\begin{equation*}
\abs{D^pu}(x)=O(\abs{\rho}^{n+1/2-p-b})
\end{equation*}
for $b>0$. In particular, we have
\begin{equation} \label{E:boundary}
\abs{u_{\alpha\bar\beta}}\le O(\abs\rho^{n-3/2-b})
\end{equation}
for $b>0$. The above discussion also implies that 
\begin{equation*}
u_{\alpha\bar\beta}\in{C}^\infty(\Omega)\cap{C}^{n-3/2-b}(\bar\Omega),
\end{equation*}
for $b>0$ and $1\le\alpha,\beta\le{n}$.

\section{Subharmonicity of K\"ahler-Einstein metrics on strongly pseudoconvex domains}

In this section, we shall discuss about Theorem \ref{T:main_theorem}. More precisely, we will prove the following:
\begin{theorem}\label{T:main_theorem'}
If every boundary point of $D_s$ is a strongly pseudoconvex boundary point of $D$, then $h$ is strictly plurisubharmonic near $D_s$.
\end{theorem}

\begin{remark}
The above theorem have been already proved if the slice dimension is greater than or equal to $3$ in \cite{Choi}. In fact, a little more is proved in \cite{Choi}. This will be discussed in Section \ref{S:remark}.
\end{remark}

\subsection{The geodesic curvature from the approximate K\"ahler-Einstein metrics}

Let $D$ be a smooth domain in $\CC^{n+1}$ such that every slice $D_s$ is strongly pseudoconvex domain. Suppose that every boundary point of $D_s$ is a strongly pseudoconvex  boundary point of the total space $D$. Then every slice $D_{s'}$ which is sufficiently close to $D_s$  has such property. Since our computation is always local in $s$-variable, we may assume that $\pi(D)=U$ and there exists a defining function $\varphi$ which satisfies the conditions in Subsection \ref{SS:GC}. By the argument in Subsection 2.3, we know that there exist approximate solutions $\varphi^{n+1}(\cdot,s)$ such that
\begin{equation} \label{E:condition_of_J}
J\paren{\varphi^{n+1}(\cdot,s)}=1+O\paren{\abs{\varphi(\cdot,s)}^{n+1}},
\end{equation}
for every $s\in U$. By \eqref{E:approximate_solution}, there exists a smooth positive function $\eta$ on $\bar D$ such that
$\varphi^{n+1}(\cdot,s)=-\eta(\cdot,s)\varphi(\cdot,s)$. Hence $\eta(\cdot,s)\varphi(\cdot,s)$ is another defining function of $D_s$ for each $s\in U$. We call it $\psi(\cdot,s)$. Since every slice $D_s$ is strongly pseudoconvex, $w=-\log(-\psi)=-\log(-\eta\varphi)$ is strictly plurisubharmonic in each slice $D_s$ when sufficiently close to the boundary. It is easy to see that $w$ can be modified away from $\partial D$ to a smooth function on $D$, which is strictly plurisubharmonic when restricted on each slice $D_s$ for $s\in U$ (by shrinking $U$, if necessary); we again denote it by $w$ (cf, see \cite{Demailly}). Now let $e^{-F}=J(-\eta\varphi)$. Then $F$ is a smooth function on $\bar D$ and satisfies that
\begin{equation*}
\det\paren{w_{\alpha\bar\beta}(z,s)}=e^{(n+1)w(z,s)}e^{-F(z,s)},
\end{equation*} Ê
and \eqref{E:condition_of_J} implies that
\begin{equation*}
F(\cdot,s)=\xi(\cdot,s)\varphi(\cdot,s)^{n+1},
\end{equation*}
for each $s\in U$, where $\xi$ is a smooth function on $\bar D$. Again $\eta$ is understood to be a smooth function on $\bar D$ such that
$w=-\log(-\eta\varphi)$. So $w_{\alpha\bar\beta}=g_{\alpha\bar\beta}-\paren{\log\eta}_{\alpha\bar\beta}$. We can solve a family of complex Monge-Amp\`{e}re equations: 
\begin{equation} \label{E:Monge-Ampere''}
	\begin{aligned}
	\det(w_{\alpha\bar\beta}(\cdot,s)+u_{\alpha\bar\beta}(\cdot,s))
		=e^{Ku(\cdot,s)}e^{F(\cdot,s)}\det(w_{\alpha\bar\beta}(\cdot,s)), \\
	\frac{1}{c}(w_{\alpha\bar\beta}(\cdot,s))
		\le(w_{\alpha\bar\beta}(\cdot,s)+u_{\alpha\bar\beta}(\cdot,s)) 
		\le{c}(w_{\alpha\bar\beta}(\cdot,s)).
	\end{aligned}
\end{equation}
We denote by $u(\cdot,s)$ the solution of \eqref{E:Monge-Ampere''} for each $s\in U$. By Theorem \ref{T:CY2} and \eqref{E:boundary}, for each slice $D_s$, we have the following boundary behavior of the solution $u$:
\begin{equation*} 
\abs{u_{\alpha\bar\beta}(\cdot,s)}\le O(\abs{\varphi(\cdot,s)}^{n-3/2-b})
\end{equation*}
for $b>0$.

Now we define a real $(1,1)$-form $W$ by $W=\im\ddbar w$. We can write $W$ as follows:
\begin{equation*}
W
=
\im\paren{w_{s\bar s}ds\wedge{d\bar s} 
+w_{s\bar\beta}ds\wedge{dz}^{\bar\beta} 
+w_{\alpha\bar s}dz^\alpha\wedge{d\bar s} 
+w_{\alpha\bar\beta}dz^\alpha\wedge{dz}^{\bar\beta}
}.
\end{equation*}
To observe the horizontal lift $v_W$ and the geodesic curvature $c(W)$, we need to compute the inverse of $w_{\alpha\bar\beta}$. 
\begin{lemma}[\cite{Choi}]\label{L:identity1}
There exists a hermitian $n\times{n}$ matrix 
\begin{equation*}
M
=
(M_{\alpha\bar\beta})\in\mathrm{Mat}_{n\times{n}}\paren{C^\infty(\bar D)},
\end{equation*}
which satisfies that
\begin{equation*}
w^{\bar\beta\alpha}-g^{\bar\beta\alpha}
=
g^{\bar\beta\gamma}M_{\gamma\bar\delta}g^{\bar\delta\alpha}.
\end{equation*}
In particular, $w^{\bar\beta\alpha}\in{C}^\infty(\bar{D})$ and $w^{\bar\beta\alpha}=O(\abs\varphi)$.
\end{lemma}
With the help of the above lemma, we can show that $v_W$ has the same properties with $v_G$.

\begin{proposition} \label{P:horizontal_lift}
Any horizontal lift $v_W$ with respect to $W$ is smoothly extended up to the boundary $\partial D$. Moreover, $v_W\vert_{\partial D}$ is tangent to  $\partial{D}$.
\end{proposition}

\begin{proof}
Note that $v_W$ is written by
\begin{equation*}
v_W
=
\pd{}{s}-w_{s\bar\beta}w^{\bar\beta\alpha}\pd{}{z^\alpha}.
\end{equation*}
Since $w=-\log(-\eta\varphi)=g-\log\eta$ and $\eta$ is smooth up to the boundary, $v_W$ is smoothly extended up to the boundary. Moreover,
\begin{align*}
v_W(\varphi)-v_G(\varphi)
&=
g_{s\bar\beta}g^{\bar\beta\alpha}\varphi_\alpha
-
w_{s\bar\beta}w^{\bar\beta\alpha}\varphi_\alpha \\
&=
g_{s\bar\beta}(g^{\bar\beta\alpha}-w^{\bar\beta\alpha})\varphi_\alpha
+
(\log\eta)_{s\bar\beta}w^{\bar\beta\alpha}\varphi_\alpha \\
&=
g_{s\bar\beta}g^{\bar\beta\gamma}M_{\gamma\bar\delta}g^{\bar\delta\alpha}\varphi_\alpha
+
(\log\eta)_{s\bar\beta}w^{\bar\beta\alpha}\varphi_\alpha \\
&=
O(\abs\varphi),
\end{align*}
this completes the proof.
\end{proof}

Recall that the geodesic curvature of $c(W)$ is given by
\begin{equation*}
c(W)
=
\inner{v_W,v_W}_W
=
w_{s\bar{s}}-w_{s\bar\beta}w^{\bar\beta\alpha}w_{\alpha\bar{s}}.
\end{equation*}
By the definition of Levi form, the geodesic curvature $c(W)$ is computed as follows:
\begin{align*}
\inner{v_W,v_W}_W
&=
\im\ddbar{w}(v_W,\overline{v_W})
=
\mathcal{L}w(v_W,\overline{v_W})\\
&=
\frac{1}{-\psi}\mathcal{L}\psi(v_W,\overline{v_W})
+\frac{1}{\psi^2}\abs{\partial\psi(v_W)}^2.
\end{align*}
\begin{remark} \label{R:Levi_form}
We can observe the following:
\begin{itemize}
\item[(1)] Since $v_W$ is tangent to $\partial D$, $\partial\varphi(v_W)\vert_{\partial D}=0$.
\item[(2)] Since $D$ is a smooth pseudoconvex  domain, $\mathcal{L}\psi(v_W,\overline{v_W})\ge0$ on $\partial D$. It follows that $c(W)\ge0$.
\item[(3)] If $D$ is strongly pseudoconvex at $p\in\partial{D}_s$, then $\mathcal{L}\psi(v_W,\overline{v_W})\vert_p>0$. It follows that 
	\begin{equation*}
	\frac{1}{-\psi(z,s)}\mathcal{L}\psi(v_W,\overline{v_W})\vert_{(z,s)}
	\rightarrow\infty
	\end{equation*}
as $(z,s)\rightarrow{p}$. In particular, $c(W)(z,s)\rightarrow\infty$ as $(z,s)\rightarrow{p}$.
\end{itemize}
\end{remark}

\subsection{Proof of Theorem \ref{T:main_theorem'}}

As we mentioned in Introduction, we denote by $h_{\alpha\bar\beta}(z,s)$ a unique complete K\"ahler-Einstein metric on a slice $D_s$. And we also denote by a function $h:D\rightarrow\RR$ defined by
\begin{equation*}
h(z,s)
=
\frac{1}{n+1}\log\det\paren{h_{\gamma\bar\delta}(z,s)}_{1\le\gamma,\delta\le{n}}.
\end{equation*}
If we define a real $(1,1)$-form $H$ by $H=\im\ddbar h$, then $H$ is a real $(1,1)$-form on D such that the restriction on each slice $D_s$ is positive-definte by the K\"ahler-Einstein condition. 
We denote by $\Delta=\Delta_{h_{\alpha\bar\beta}}$ the Laplace-Beltrami operator with respect to the K\"ahler-Einstein metric $h_{\alpha\bar\beta}$ on $D_s$. Schumacher proved that the geodesic curvature $c(H)$ of $H$ satisfies a certain elliptic partial differential equation on each slice. (For the proof, see \cite{Schumacher3} or \cite{Choi}.)

\begin{theorem} [\cite{Schumacher3}]
The following elliptic equation holds slicewise:
\begin{equation} \label{E:PDE}
-\Delta c(H)+(n+1)c(H)
=
\abs{\bar\partial v_H}^2.
\end{equation}
\end{theorem}

From now on, we fix a slice $D_s$ and we think the geodesic curvatures $c(W)$ and $c(H)$ as functions on $D_s$. By the hypothesis, every boundary point of $D_s$ is a strongly pseudoconvex boundary point of $D$. It follows that $c(W)\rightarrow\infty$ as $x\rightarrow \partial D_s$ by Remark \ref{R:Levi_form}. The following proposition is describe the boundary behavior of $c(H)$ in terms of $c(W)$.
\begin{proposition}\label{P:ratio_boundary_behavior}
The geodesic curvatures $c(W)$ and $c(H)$ go to infinity near the boundary of the same order. More precisely, we have
\begin{equation}\label{E:boundedness}
\frac{c(H)}{c(W)}(x)\rightarrow1
\;\;\;
\textrm{as}
\;\;\;
x\rightarrow\partial D_s.
\end{equation}
\end{proposition}

In the next subsection, we shall prove Proposition \ref{P:ratio_boundary_behavior}. In a moment, assuming that, we want to complete the proof.
\medskip

From \eqref{E:boundedness} we know that $c(H)$ is bounded from below. Then we can apply the almost maximum principle due to Yau (\cite{Yau}), namely, there exists a sequence $\{x_k\}_{k\in\NN}\subset D_s$ such that
\begin{align*}
&\lim_{k\rightarrow\infty}\nabla c(H)(x_k)=0,\;\;
\liminf_{k\rightarrow\infty}\Delta c(H)(x_k)\ge0,\;\;
\text{and}\\
&
\hspace{1.5cm}
\lim_{k\rightarrow\infty}c(H)(x_k)
=\inf_{x\in D_s}c(H)(x).
\end{align*}
It follows that
\begin{equation*}
(n+1)c(H)(x_k,y)
=
\abs{\bar\partial{v}_H}^2+\Delta c(H)(x_k,y)>0.
\end{equation*}
Taking $k\rightarrow\infty$, we have $c(H)\ge0$.

We also know that $c(H)\rightarrow\infty$ as $x\rightarrow\partial D_s$ by \eqref{E:boundedness}. But this prevents the function $c(H)$ from being zero. In fact, according to a theorem of Kazdan and De Turck (\cite{DeTurck_Kazdan}), K\"{a}hler-Einstein metrics are real analytic on holomorphic coordinates, and by the Implicit Function Theorem, depend in a real-analytic way upon holomorphic parameters. This also applies to the function $c(H)$.

\begin{proposition} \label{P:uniqueness}
Let $\omega$ be a K\"ahler form in $\CC^n$. Let $f$ and $g$ be non-negative smooth functions on $U\subset\CC^n$. Suppose
\begin{equation*}
-\Delta_\omega f+Cf
=g
\end{equation*}
holds for some positive constant $C$. If $f(0)=0$, then $f$ and $g$ vanish identically in a neighborhood of $0\in\CC^n$.
\end{proposition}

\begin{proof}
It follows from the assumption that $\psi$ has a local minimum at the origin, and (3) implies that $\Delta_{\omega_U}\psi(0)=0$ and $f(0)=0$.

We set $\Delta=\Delta_{\omega_U}$ and choose normal coordinates $z^\alpha$ of the second kind for $\omega_U$ at $0$. Let $\Delta_0=\sum_{\alpha=1}^n\pd{^2}{z^\alpha\partial{z}^{\bar\alpha}}$ be the standard Laplacian so that
\begin{equation*}
-\Delta=-\Delta_0+t^{\bar\beta\alpha}\pd{^2}{z^\alpha\partial{z}^{\bar\beta}}
\end{equation*}
where the power series expansion of all $t^{\bar\beta\alpha}$ have no terms of order zero or one. Then the maximum principle of E. Hopf implies that $\psi\equiv0$. (cf. See Theorem 6, Chap. 2, Sect. 3 in \cite{Protter_Weinberger}.)
\end{proof}

The real analyticity of $c(H)$ and Proposition \ref{P:uniqueness} say that $c(H)$ is either identically zero, or never zero. However we know that $c(W)(x)\rightarrow\infty$ as $x\rightarrow\partial D_s$. This completes the proof.

\subsection{The boundary behavior of $c(H)$} \label{SS:bb}
In this subsection, we shall prove Proposition \ref{P:ratio_boundary_behavior}.
\medskip

Recall that Remark \ref{R:Levi_form} says that $c(W)$ is given by the following:
\begin{equation*}
c(W)
=
\frac{1}{-\psi}\mathcal{L}\psi(v_W,\overline{v_W})
+\frac{1}{\psi^2}\abs{\partial\psi(v_W)}^2.
\end{equation*}
Since every boundary point of $D_s$ is a strongly pseudoconvex boundary point of $D$ and $v_W$ is tangent to $\partial D$, we have
\begin{equation*}
c(W)\ge C\cdot\abs\psi
\end{equation*}
for some constant $C>0$ when a point goes to $\partial D_s$, in particular $c(W)$ blows up of order greater than or equal to $1$. To compute $c(H)$ in terms of $c(W)$, we need the following lemma.

\begin{lemma}[\cite{Choi}]\label{L:identity2}
For each $s\in U$, there exists a hermitian $n\times{n}$ matrix 
\begin{equation*}
N^s
=
(N^s_{\alpha\bar\beta})
\in
\mathrm{Mat}_{n\times{n}}\paren{C^\infty(D_s)\cap{C}^{n-3/2-b}(\bar D_s)}
\end{equation*}
with $\norm{N^s}=O(\abs{\varphi(\cdot,s)}^{n-3/2-b})$ for $b>0$, which satisfies that
\begin{equation*}
h^{\bar\beta\alpha}(\cdot,s)-w^{\bar\beta\alpha}(\cdot,s)
=
w^{\bar\beta\gamma}(\cdot,s)N^s_{\gamma\bar\delta}w^{\bar\delta\alpha}(\cdot,s).
\end{equation*}
In particular, $h^{\bar\beta\alpha}(\cdot,s)\in{C}^\infty(D_s)\cap{C}^{n-3/2-b}(\bar D_s)$ for $b>0$ and $h^{\bar\beta\alpha}(\cdot,s)=O(\abs{\varphi(\cdot,s)})$.
\end{lemma}

By Lemma \ref{L:identity2}, $c(H)$ is computed as follows:
\begin{align*}
c(H)
&=
h_{s\bar{s}}-h_{s\bar\beta}h^{\bar\beta\alpha}h_{\alpha\bar{s}} \\	
&=
h_{s\bar{s}}-h_{s\bar\beta}
\paren{w^{\bar\beta\alpha}
+
w^{\bar\beta\gamma}N_{\gamma\bar\delta}w^{\bar\delta\alpha}
}
h_{\alpha\bar{s}} \\
&=
w_{s\bar{s}}+u_{s\bar{s}}
-
\paren{w_{s\bar\beta}+u_{s\bar\beta}}\paren{w^{\bar\beta\alpha}
+
w^{\bar\beta\gamma}N_{\gamma\bar\delta}w^{\bar\delta\alpha}
}
\paren{w_{\alpha\bar{s}}+u_{\alpha\bar{s}}} \\
&=
c(W) + (\text{remaining terms}),
\end{align*}
where the remaining terms are given by the sum of 
\begin{equation*}
R_1:=
	u_{s\bar{s}}+\paren{w_{s\bar\beta}w^{\bar\beta\alpha}u_{\alpha\bar{s}}
	+u_{s\bar\beta}w^{\bar\beta\alpha}w_{\alpha\bar{s}}
	+u_{s\bar\beta}w^{\bar\beta\alpha}u_{\alpha\bar{s}}
	}
\end{equation*}
and
\begin{equation*}
R_2
	:=\paren{w_{s\bar\beta}+u_{s\bar\beta}}
	w^{\bar\beta\gamma}N_{\gamma\bar\delta}w^{\bar\delta\alpha}
	\paren{w_{\alpha\bar{s}}+u_{\alpha\bar{s}}}.
\end{equation*}
Hence it is enough to show that 
\begin{equation*}
\frac{R_1+R_2}{c(W)}\rightarrow0
\;\;\;
\textrm{as}
\;\;\;
x\rightarrow \partial D_s.
\end{equation*}
First we note that $u_{s\bar s}$ is bounded by Section 3 in \cite{Choi}. By taking logarithm of \eqref{E:Monge-Ampere''} and differentiating it with respect to $s$, we know that $u_s$ satisfies the following linear elliptic partial differential equation on each slice $D_s$:
\begin{equation} \label{E:linear_equation}
-\Delta u_s+(n+1)u_s=Q,
\end{equation}
where $Q=-F_s+\paren{\Delta-\Delta_{w_{\alpha\bar\beta}}}w_s$. Here $\Delta_{w_{\alpha\bar\beta}}$ is the Laplace-Beltrami operator with respect to the K\"ahler metic $w_{\alpha\bar\beta}$. Note that the boundary behavior of the solution of complex Monge-Amp\`ere equation implies that
\begin{equation}\label{E:estimate_R}
Q=O(\abs{\varphi}^{n-3/2-b})
\end{equation}
for $b>0$. We need the following lemma.

\begin{proposition} \label{P:second}
Let $0<r\le 1$. If $Q=O(\abs\varphi^r)$, then $\abs{u_s}=O(\abs\varphi^r)$.
\end{proposition}

\begin{proof}
In case of $r=1$, it is proved in \cite{Choi}. Thus we may assume that $0<r<1$. 
For $c>0$, we compute
\begin{align*}
\Delta (u_s-c(-\varphi)^r)
&=
h^{\alpha\bar\beta} 
(u_s-c(-\varphi)^r)_{\alpha\bar\beta} \\
&= 
h^{\alpha\bar\beta}(u_s)_{\alpha\bar\beta}
-
h^{\alpha\bar\beta}(c(-\varphi)^r)_{\alpha\bar\beta} \\
&=
(n+1)(u_s)-Q
-h^{\alpha\bar\beta}\paren{(cr(-\varphi)^{r-1})(-\varphi)_\alpha}_{\bar\beta} \\
&=
(n+1)(u_s)-Q
-cr(r-1)(-\varphi)^{r-2}h^{\alpha\bar\beta}(-\varphi)_\alpha(-\varphi)_{\bar\beta}\\
&
\;\;\;\;+cr(-\varphi)^{r-1}h^{\alpha\bar\beta}\varphi_{\alpha\bar\beta}.
\end{align*}
Since $h^{\alpha\bar\beta}$ is positive definite and $\varphi$ is plurisubharmonic, we know that  
\begin{equation*}
h^{\alpha\bar\beta}(-\varphi)_\alpha(-\varphi)_{\bar\beta}>0
\quad\text{and}\quad
h^{\alpha\bar\beta}\varphi_{\alpha\bar\beta}>0.
\end{equation*}
It follows that 
\begin{equation*}
-cr(r-1)(-\varphi)^{r-2}h^{\alpha\bar\beta}(-\varphi)_\alpha(-\varphi)_{\bar\beta}+cr(-\varphi)^{r-1}h^{\alpha\bar\beta}\varphi_{\alpha\bar\beta}>0.
\end{equation*}
So we have
\begin{equation*}
\Delta(u_s-c(-\varphi)^r)
\ge
(n+1)(u_s)-Q.
\end{equation*}
Since $Q=O(\abs\varphi^r)$, we can choose $c_1>0$ such that
\begin{align*}
\Delta(u_s-c_1(-\varphi)^r)
\ge
(n+1)(u_s-c_1(-\varphi)^r).
\end{align*}
Note that $u_s$ is bounded by Section 3 in \cite{Choi}. The almost maximum principle of Yau (\cite{Yau}) implies that $u_s-c_1(-\varphi)^r\le0$ , i.e., $u\le c_1(-\varphi)^r$ in $D_s$.

If we apply the same argument to $\Delta(u_s+c(-\varphi)^r)$ for $c>0$, then we have that $u_s\ge -c_2(\abs\varphi^r)$ for some constant $c_2>0$.  Therefore $u_s=O(\abs\varphi^r)$ as desired.
\end{proof}

Let $(V,(v^1,\dots,v^n))$ be a coordinate system in $\Omega$ satisfying the conditions in Definition 1.1 in \cite{Cheng_Yau}. (cf, see \cite{Choi}.) Note that every bounded smooth strongly pseudoconvex domain in $\CC^n$ admits a open covering of such coordinates. (This is constructed in Section 1 in \cite{Cheng_Yau}). For a smooth function $f$, we write
\begin{equation*}
\begin{aligned}
\abs{f}_{k+\varepsilon,V}
	& =\sup_{z\in{V}}\paren{
	\sum_{\abs{\alpha}+\abs{\beta}\le{k}}\abs{\pd{^{\abs{\alpha}+\abs{\beta}}}
	{v^\alpha\partial{v}^{\bar\beta}}f(z)}
	} \\ 
	&\;\;\; + \sup_{z,z'\in{V}}\paren{
	\sum_{\abs\alpha+\abs\beta={k}}\abs{z-z'}^{-\varepsilon}\abs{\pd{^{\abs\alpha+\abs\beta}}
	{v^\alpha\partial{v}^{\bar\beta}}f(z) - \pd{^{\abs\alpha+\abs\beta}}{v^\alpha\partial{v}^{\bar\beta}}f(z')}
	},
\end{aligned}
\end{equation*}
where $k$ is non-negative integer and $\varepsilon\in(0,1)$. 

Applying the Schauder estimates to Equation \eqref{E:linear_equation} in the coordinate system $(V,(v^1,\dots,v^n))$, we obtain that
\begin{equation} \label{E:Schauder}
\abs{u_s}^*_{2+\varepsilon,V}
\le
{C}\paren{\abs{u_s}_{0,V}
+
\abs{Q}^{(2)}_{0+\varepsilon,V}}.
\end{equation}
(For detailed notations, we refer to see \cite{Gilbarg_Trudinger}.) Instead of introducing the definitions of $\abs\cdot^*_{k+\varepsilon,V}$ and $\abs\cdot^{(2)}_{k+\varepsilon,V}$, we note that the construction of the coordinate system $(V,(v^1,\dots,v^n))$ implies that there exists a open subset $V'$ of $V$ and a uniform constant $C>0$ such that
\begin{align*}
\frac{1}{C}\abs{f}_{k+\varepsilon,V'}
&<
\abs{f}^*_{k+\varepsilon,V'}
<
C\abs{f}_{k+\varepsilon,V'}
\quad\text{and}\quad 
\\
\frac{1}{C}\abs{f}_{k+\varepsilon,V'}
&<
\abs{f}^{(2)}_{k+\varepsilon,V'}
<
C\abs{f}_{k+\varepsilon,V'}.
\end{align*}
By \eqref{E:estimate_R}, if $\varepsilon>0$ is sufficiently small, then
$$
\abs{Q}_{0+\varepsilon,V}=O(\abs\varphi^{n-3/2-b'}).
$$
for some $b'$ with $0<b'<1/2$. 
This together with Proposition \ref{P:second} implies that
\begin{equation}\label{E:bdry}
\abs{u_s}_{0,V}=
\begin{cases}
O(\abs\varphi)\;\;\;\text{if}\;\;\;n\ge3, \\
O(\abs\varphi^{1/2-b'})\;\;\;\text{if}\;\;\;n=2. \\
\end{cases}
\end{equation}
It follows that $\abs{u_s}_{2+\varepsilon,V'}=O(\abs\varphi^{1/2-b'})$. Hence we have 
\begin{equation*}
\abs{u_s}_{2+\varepsilon,V'}<C.
\end{equation*}
for some constant $C>0$. In particular, we have $\abs{u_s}_{1,V'}<\infty$. By the construction of coordinate system $(V,(v^1,\dots,v^n))$ on a bounded smooth strongly pseudoconvex domain (see Section 1 in \cite{Cheng_Yau}), we know that
\begin{equation*}
\sup_{V'}\abs{\sum_\alpha\pd{u_s}{z^\alpha}}
\le
\frac{C}{\abs\varphi}\abs{u_s}_{1,V'},
\end{equation*}
for some uniform constant $C>0$. Hence we have
\begin{equation*}
\sup_{V'}\abs{\sum_\alpha\pd{u_s}{z^\alpha}}
\le
C(\abs\varphi^{-1/2-b'}).
\end{equation*}
Together with Lemma \ref{L:identity1} and Lemma \ref{L:identity2}, it follows that $R_2$ is bounded. Moreover, it also implies that
\begin{equation*}
\frac{R_1}{c(W)}(x)\rightarrow0
\quad\text{as}\quad
x\rightarrow\partial D_s.
\end{equation*}
Hence we have
\begin{equation*}
\frac{c(H)}{c(W)}(x)\rightarrow1
\;\;\;
\textrm{as}
\;\;\;
x\rightarrow \partial D_s.
\end{equation*}
This completes the proof of Proposition~\ref{P:ratio_boundary_behavior}.

\section{Proof of Corollary \ref{C:subharmonicity}}
In this section, we discuss about the variations of K\"ahler-Einstein metrics on bounded pseudoconvex domains. First we discuss about the construction of the K\"ahler-Einstein metric on a bounded pseudoconvex domain. And we prove Corollary \ref{C:subharmonicity} in the next subsection.

\subsection{K\"ahler-Einstein metric on a bounded pseudoconvex domain}
Let $\Omega$ be a bounded pseudoconvex domain.
Then there exists a smooth strictly plurisubharmonic exhaustion function $\psi$. For $N\in\NN$, we denote by
\begin{equation*}
\Omega^N
=
\{z\in D:\psi(z)<N\}.
\end{equation*}
By Sard theorem, we may assume that $\Omega^N$ is a bounded strongly pseudoconvex domain with smooth boundary. It is also obvious that $\{\Omega^N\}$ is a increasing union of $\Omega$. Then the theorem of Cheng and Yau implies that  there exists a unique complete K\"ahler-Einstein metric $h^N_{\alpha\bar\beta}$ on $\Omega^N$ with Ricci curvature $-(n+1)$. By the Schwarz lemma for volume form due to Mok and Yau (\cite{Mok_Yau}), we have that $\{\det(h^N_{\alpha\bar\beta})\}$ is a decreasing sequence, more precisely,
\begin{equation*}
\det(h^N_{\alpha\bar\beta})
\ge
\det(h^{N'}_{\alpha\bar\beta})
\;\;\;\text{for}\;\;\;
N<N'.
\end{equation*}
From the K\"ahler-Einstein condition, $\log\det(h^N_{\alpha\bar\beta})$ is a strictly plurisubharmonic function on $\Omega^N$. It follows that $\left\{\log\det(h^N_{\alpha\bar\beta})\right\}_{N\in\NN}$ is a decreasing sequence of plurisubharmonic functions. This implies that the sequence converges to a plurisubharmonic function $h$. It is proved that $h_{\alpha\bar\beta}$ is the unique complete K\"ahler-Einstein metric with Ricci curvature $-(n+1)$ by Cheng-Yau and Mok-Yau.

\subsection{Plurisubharmonicity of the variations}

Let $D$ be a bounded pseudoconvex domain in $\CC^{n+1}$ such that every slice $D_s$ is a bounded pseudoconvex domain. By the theorem of Mok and Yau, there exists a unique complete K\"ahler-Einstein metric $h_{\alpha\bar\beta}(z,s)$ with Ricci curvature $-(n+1)$. If we define the function $h:D\rightarrow\RR$ by
\begin{equation*}
h(z,s)
=
\frac{1}{n+1}\log\det\paren{h_{\gamma\bar\delta}(z,s)}_{1\le\gamma,\delta\le{n}}
\end{equation*}
then $h$ is strictly plurisubharmonic on each slice $D_s$. Since $D$ is a pseudoconvex domain, there exists a strictly plurisubharmonic exhaustion function $\psi$ on $D$. Let 
$D^N=\{(z,s)\in\CC^{n+1}:\psi(z,s)<N\}$ for $N\in\NN$. Then we have the following:
\begin{itemize}
\item $D^N\subset\subset D$ and $D$ is increasing union of $\{D^N\}$,
\item each $D^N$ is a bounded smooth strongly pseudoconvex subdomain in $D$.
\end{itemize}
Denote by $D_s^N=D^N\cap D_s$. Then there exists a unique complete  K\"ahler-Einstein metric $h^N_{\alpha\bar\beta}(z,s)$ on each $D_s^N$. Define a function $h^N:D^N\rightarrow\RR$ by 
\begin{equation*}
h^N(z,s)
=
\frac{1}{n+1}\log\det\paren{h^N_{\gamma\bar\delta}(z,s)}_{1\le\gamma,\delta\le{n}}
\end{equation*}
for every $N\in\NN$. Then we know that $h^N$ is a smooth strictly plurisubharmonic function on $D^N$ by Section 3 (cf, see \cite{Choi}). On each slice $D_s$, $h^N(\cdot,s)$ forms a decreasing sequence which converges to $h(\cdot,s)$. It follows that the sequence $\{h^N\}$ is a decreasing sequence which converges to $h$ on $D$. This implies that $h$ is limit of a decreasing sequence of plurisubharmonic functions, in particular $h$ is plurisubharmonic.

\section{Local trivility}
In this section, we discuss about the local triviality of a family of smooth bounded strongly pseudoconvex domains. 

Let $D$ be a smooth domain in $\CC^{n+1}$ such that every slice $D_s$ is a bounded strongly pseudoconvex domain with smooth boundary. Since the computation is local, we may assume that $\pi(D)=U$ the standard unit disc in $\CC$. Suppose that the geodesic curvature $c(H)$ of $H$ vanishes in $D$. Then \eqref{E:PDE} implies that $\abs{\bar\partial v_H}$ vanishes, i.e., $v_H$ is a holomorphic vector field on $D$. Thus we have a holomorphic vector field $v_H$ on $D$ such that $d\pi(v_H)=\partial/\partial s$.

\begin{proposition} \label{P:horizontal_lift'}
On each slice $D_s$, the horizontal lift $v_H$ is extended continuously up to the boundary $\partial D_s$ and it is tangent to the boundary $\partial D$. More precisely, the following holds:
\begin{equation*}
v_H(\varphi)=O(\abs\varphi^r)
\end{equation*}
for some $0<r<1$.
\end{proposition}

\begin{proof}
Note that the horizontal lift $v_H$ of $\partial/\partial s$ with respect to $H$ is given by
\begin{equation*}
v_H=
\pd{}{s}-h_{s\bar\beta}h^{\bar\beta\alpha}\pd{}{z^\alpha},
\end{equation*}
where $h_{s\bar\beta}$ and $h^{\bar\beta\alpha}$ are 
\begin{equation*}
h_{s\bar\beta}
=
w_{s\bar\beta}+u_{s\bar\beta}
\;\;\;
\textrm{and}
\;\;\;
h^{\bar\beta\alpha}
=
w^{\bar\beta\alpha}
+
w^{\bar\beta\gamma}N_{\gamma\bar\delta}w^{\bar\delta\alpha}.
\end{equation*}
By Lemma \ref{L:identity1}, we already know that $w^{\bar\beta\alpha}$ is smooth up to the boundary and $w^{\bar\beta\alpha}=O(\abs\varphi)$. It is easy to see that $w_{s\bar\beta}w^{\bar\beta\alpha}$ is smooth up to the boundary. Moreover the boundary behavior of $u_s$ implies $u_{s\bar\beta}=O(\abs\varphi^{n-5/2-b'})$ (note that we can choose $b'>0$ sufficiently small; see Subsection \ref{SS:bb}). Since $n\ge2$, all together implies the first assertion.

To show the second assertion, we compute $v_H(\varphi)$. We already know that $v_W(\varphi)=O(\abs\varphi)$. 
\begin{align*}
v_H(\varphi)
&=
	h_{s\bar\beta}h^{\bar\beta\alpha}\varphi_\alpha \\
&=
	\paren{w_{s\bar\beta}+u_{s\bar\beta}
	}
	\paren{w^{\bar\beta\alpha}
		+w^{\bar\beta\gamma}N_{\gamma\bar\delta}w^{\bar\delta\alpha}
	}
	\varphi_\alpha \\
&=
	v_W(\varphi)
	-w_{s\bar\beta}w^{\bar\beta\gamma}
		N_{\gamma\bar\delta}w^{\bar\delta\alpha}\varphi_\alpha
	-u_{s\bar\beta}w^{\bar\beta\gamma}
		N_{\gamma\bar\delta}w^{\bar\delta\alpha}\varphi_\alpha
	-u_{s\bar\beta}w^{\bar\beta\gamma}\varphi_\alpha.
\end{align*}
Obviously $v_W=O(\abs\varphi)$ by the proof of Proposition \ref{P:horizontal_lift}. Lemma \ref{L:identity2} implies that the second and third terms are also $O(\abs\varphi)$. By \eqref{E:bdry} implies that the last term satisfies that 
\begin{equation*}
\abs{u_{s\bar\beta}w^{\bar\beta\gamma}\varphi_\alpha}=
\begin{cases}
O(\abs\varphi)\;\;\;\text{if}\;\;\;n\ge3, \\
O(\abs\varphi^{1/2-b'})\;\;\;\text{if}\;\;\;n=2. \\
\end{cases}
\end{equation*}
Hence we have
\begin{equation}\label{E:tangentiality}
\abs{v_H(\varphi)}=
\begin{cases}
O(\abs\varphi)\;\;\;\text{if}\;\;\;n\ge3, \\
O(\abs\varphi^{1/2-b'})\;\;\;\text{if}\;\;\;n=2. \\
\end{cases}
\end{equation}
This yields the conclusion.
\end{proof}

\begin{proposition}\label{P:bijective}
Suppose that $\abs{v_H\varphi}<c\abs\varphi$ for some $c>0$. Then the flow of $v_H$ gives a biholomorphism from $D_0\times U$ to $D$.
\end{proposition}

\begin{proof}
Let $p$ be a point in $D_0$. Define $\alpha:(a_p,b_p)\rightarrow D$ by a flow of $v_H$ passing through $p$, i.e.,
\begin{equation*}
\frac{d}{dt}\alpha(t)
=
v_H\vert_{\alpha(t)}
\quad\text{and}\quad
\alpha(0)=p.
\end{equation*}
We assume that $(a_p,b_p)$ is maximal. Now we claim that $a_p=-1$ and $b_p=1$. Define a function $f:(a_p,b_p)\rightarrow\RR$ by $f(t)=(\varphi\circ\alpha)(t)$. Then the hypothesis implies that $\abs{f'(t)}<c\abs{f(t)}$. It follows that
\begin{equation*}
-c<\frac{f'(t)}{f(t)}<c.
\end{equation*}
Integrating this, we have
\begin{equation*}
\int_0^\tau -c\;d\tau<\int_0^\tau\frac{f'(t)}{f(t)}d\tau<\int_0^\tau c\;d\tau,
\end{equation*}
i.e.,
\begin{equation*}
e^{-c\tau}<\abs{\frac{f(\tau)}{f(0)}}<e^{c\tau}.
\end{equation*}
Since $f(0)=(\varphi\circ\alpha)(0)=\varphi(p)<0$, this implies that
\begin{equation*}
\varphi(p)e^{-c\tau}
<
(\varphi\circ\alpha)(\tau)
<
\varphi(p)e^{c\tau},
\end{equation*}
for $\tau\in(-1,1)$.
Since $D$ is a fibration over $U$, it is obvious that $a_p=-1$ and $b_p=1$.
Hence by integrating the holomorphic vector field $v_H$, we obtain the biholomorphism from $D_0\times U$ to $D$.
\end{proof}

Hence Proposition \ref{P:bijective} and \eqref{E:tangentiality} imply the following theorem.

\begin{theorem}
Suppose that the slice dimension $n$ is greater than or equal to $3$. If the geodesic curvature $c(H)$ vanishes on $D$, then $D$ is biholomorphic to $D_0\times U$.
\end{theorem}

\begin{proof}
Note that the constant $C_s$ from \eqref{E:tangentiality} depends on $s$, i.e., we have $C_s>0$ such that 
\begin{equation*}
v_H(\varphi)=C_s(\abs\varphi).
\end{equation*}
This constant $C_s$ is coming from \eqref{E:Schauder}. Hence it is enough to show that there exists a constant $C$ in \eqref{E:Schauder} which does not depend on $s$. By the Schauder theorem, the constant $C$ depends only on the $n, \varepsilon, \Lambda$ where $\Lambda$ satisfies that
\begin{equation*}
h^{\alpha\bar\beta}(z,s)\xi^\alpha\xi^{\bar\beta}
\ge
\Lambda\abs\xi^2
\;\;\;
\text{for}
\;\;\;
z\in V,
\;\;\;
\xi\in\CC^n.
\end{equation*}
We have a uniform lower bound of $\Lambda$ which does not depend on $s$ because of the following:
\begin{itemize}
\item [1.] $(V,(v^1,\dots,v^n))$ is a special coordinate constructed by Cheng and Yau. On this coordinate, the metric tensor $h_{\alpha\bar\beta}$ with respect to $(v^1,\dots,v^n)$ is uniformly equivalent to the Euclidean metric, i.e., there exists a uniform constant $c>0$ such that
\begin{equation*}
\frac{1}{c}\delta_{\alpha\bar\beta}
<
h_{\alpha\bar\beta}
<
c\delta_{\alpha\bar\beta}.
\end{equation*}
\item [2.] The construction of $(V,(v^1,\dots,v^n))$ is algebraic (just using linear fractional transforms), in particular, if the strongly pseudoconvex domain varies smoothly, then the coordinates also varies smoothly. Hence we can choose the uniform constants $R, c, \mathscr{A}_l$ in Definition 1.1 in \cite{Cheng_Yau}, which do not depend on $s$.
\end{itemize}
Therefore we have the conclusion by Proposition \ref{P:bijective}.
\end{proof}

\section{A remark on $2$-dimensional slice case}\label{S:remark}
In this section we discuss about the difference between 2-dimensional case and higher dimensional case.

Together with the computation of \cite{Choi}, we have already seen the following:
\begin{itemize}
\item [(\romannumeral1)]$\abs{c(H)-c(W)}$ is bounded if the slice dimension $n\ge3$.
\item [(\romannumeral2)]$\displaystyle\frac{c(H)}{c(W)}\rightarrow1$  as $x$ goes to a strongly pseudoconvex boundary point.
\end{itemize}

If the slice dimension is equal to or greater than $3$, then (\romannumeral1) implies that $c(H)$ is bounded from below in a fixed slice $D_s$. By applying the almost maximum principle to \eqref{E:PDE}, it follows that $c(H)$ is nonnegative, i.e., the function $h$ is plurisubharmonic. Hence we have that if the boundary of $D_s$ has a strongly pseudoconvex boundary point in $D$, then $c(H)$ is strict positive, namely, $h$ is strictly plurisubharmonic.

On the other hand, if the slice dimension is equal to $2$, then we do not know whether $c(H)$ is bounded from below. We only know that $c(H)$ goes to the infinity if the point goes to the strongly pseudoconvex boundary point of $D$. Hence we cannot draw the conclusion that $c(H)$ is nonnegative provided that $D_s$ has a strongly pseudoconvex boundary point of $D$. However, if every boundary point of $D_s$ is a strongly pseudoconvex boundary point of $D$, then $c(H)$ is bounded from below. Again the almost maximum principle implies that $c(H)$ is nonnegative. Then we have that $c(H)$ is strictly positive by Proposition \ref{P:uniqueness}, i.e., $h$ is strictly plurisubharmonic. Therefore, it is quite natural to ask the following question:

\begin{question}
Let $D$ be a pseudoconvex domain in $\CC^{n+1}$ with smooth boundary. Suppose that there exists a boundary point $p$ of $D_s$ such that $p$ is a strongly pseudoconvex boundary point of $D$. Is $h$ is strictly plurisubharmonic near $D_s$?
\end{question}

\noindent\textbf{Acknowledgements.} 
First of all, the author would like to thank Professor Jun-Muk Hwang for suggesting this problem, sharing his ideas, and for his constant support. He would also like to thank Bo Berndtsson, Mihai P\v{a}un and Georg Schumacher for many helpful advices and discussions and thank Xu Wang for helpful comments about the local triviality. The author was supported by TJ Park Science Fellowship funded by POSCO TJ Park Foundation.

\end{document}